\documentclass{amsart}
\usepackage{enumitem}
\usepackage{amssymb,amsmath}
\usepackage{xfrac} 
\usepackage{faktor}
\usepackage{graphicx}
\usepackage{float,enumitem}

\usepackage{tikz}
     \usetikzlibrary{calc,intersections,through,backgrounds}
     \usepackage{tkz-euclide}

\parskip=\smallskipamount

\newtheorem{theorem}{Theorem}
\newtheorem{lemma}[theorem]{Lemma}

\newtheorem{proposition}[theorem]{Proposition}

\theoremstyle{definition}
\newtheorem{definition}[theorem]{Definition}
\newtheorem{example}[theorem]{Example}
\newtheorem{remark}[theorem]{Remark}
\newtheorem{question}[theorem]{Question}

\newcommand{\C}{\mathbb{C}}

\newcommand{\N}{\mathbb{N}}

\newcommand{\R}{\mathbb{R}}

\newcommand{\Conv}{{\mathrm{Conv}}}
\newcommand{\Conf}{{\mathrm{Conf}}}
\newcommand{\UConf}{{\mathrm{UConf}}}
\newcommand{\Symm}{{\mathrm{Symm}}}

\newcommand{\cP}{\mathcal{P}}

\newcommand\dist{\mathrm{dist}}

\newcommand{\PO}{P^{Oka}}
\newcommand{\CP}{\mathbb{CP}}

\def\bar{\overline}

\numberwithin{equation}{section}

\begin{document}
\title[Polynomial convexity]
{Polynomial convexity with degree bounds}
\author[Marko Slapar]{Marko Slapar}
\address{Faculty of mathematics and physics\\University of
  Ljubljana\\Jadranska 19\\1000
  Ljubljana, Slovenia\\ \newline
  Faculty of education\\University of
  Ljubljana\\Kardeljeva plo\v s\v cad 16\\1000
  Ljubljana, Slovenia\\ \newline
  Institute of Mathematics, Physics and Mechanics\\Jadranska
  19\\1000 Ljubljana, Slovenia}
\email{marko.slapar@fmf.uni-lj.si}

\thanks{Supported by the research program P1-0291 and research projects J1-3005 and N1-0237 at the Slovenian Research and Innovation Agency.}

%
%
\date{\today} 

\begin{abstract}  We introduce different notions of polynomial convexity with bounds on degrees of polynomials in $\C^n$. We provide some examples in higher dimensions and show necessary and sufficient conditions for polynomial convexity with degree bounds for certain sets of points in $\C$ and for certain arcs in the unit circle.
\end{abstract}
\maketitle

\section{Introduction} 

\noindent For a compact set $K\subset \C^n$, its polynomial
  hull is defined as \[\cP(K)=\{z\in\C^n; |P(z)|\le \|P\|_K \text{
    for all polynomials }P\}.\] The set is called polynomially
  convex, if $P(K)=K$. We can characterise polynomial convexity in the following way. Let
$F:[s,\infty)\times \C^n\to\C$ be a continuous map, such that for
every $t\in [s,\infty)$, the map $$F_t=F(t,\cdot)\!:\C^n\to\C$$ is a
polynomial, and let for every $t\in [s,\infty)$ the $$H_t=\{z\in\C^n;
F_t(z)=0\}$$ be the zero set of $F_t$.  We call the map $t\mapsto H_t$
the \textit{curve of algebraic hypersurfaces} in $\C^n$. If
$$\lim_{t\to\infty}\dist(0,H_t)=\infty,$$ we say that $H_t$
\textit{tends to infinity}. The following result was proved by
Stolzenberg (\cite{St}) and is often referred as the Oka
characterisation of polynomial hulls.
\begin{proposition}[\cite{St}, \&1] A compact set $K\subset\C^n$ is polynomially convex, if and only if, for every $z_0\in\C^n\backslash K$, there exists a curve of algebraic surfaces $H_t$, $t\ge s$, that tends to infinity, so that $z_0\in H_s$ and $H_t\cap K=\emptyset$ for every $t\ge s $.
\end{proposition}  

\noindent The proof in one direction is immediate. If  $z_0\in \C^n\backslash\cP(K)$
and  $P:\C^n\to\C$ is such that $P(z_0)=1$ and $\|P\|_K<1$, then
$H_t=\{P(z)=t\}$, $t\ge 1$, gives a curve of algebraic surfaces in the
complement of $K$, that tends to infinity, and $z_{0}\in H_0$. The other
direction requires approximating $1/F_t$ on $\cP(K)$ by polynomials
and does not give such a clear correspondence between $F_t$ and the
polynomial $P$ that separates $z_0$ from $\cP(K)$. Notice that one can
already take the curve of algebraic surfaces to be the real level sets of
a single polynomial.  

By restricting the degrees of the polynomial, this motivates us to
define the following polynomial hulls with restricted degrees.

\begin{definition} Let $K\subset\C^n$ be a compact set and let
  $d\in\N\cup\{\infty\}$.
    \begin{itemize}[leftmargin=*]
\item   The \textit{d-polynomial hull} of $K$ is the set
  \[\cP_d(K)=\{z\in\C^n; |P(z)|\le \|P\|_K \text{
      for all polynomials }p \text{ with }\deg P\le d\}.\] A compact
  set $K\subset\C^n$ is\textit{ d-polynomially convex} if $\cP_d(K)=K$.
 \item The \textit{Oka d-polynomial hull} of $K$ is the set
  $\cP^g_d(K)$ containing all the points $z_0\in\C^n$, such that every curve
  of algebraic surfaces $H_t$, $t\ge s$, that tends to infinity, with
  $z_0\in V_0$ and $\deg F_t\le d$, where $F:[s,\infty)\times
  \C^n\to\C$ is the function defining the curve of algebraic surfaces,
  has a nonempty intersection with $K$.  A compact
  set $K\subset\C^n$ is\textit{ Oka d-polynomially convex} if $\cP^{Oka}_d(K)=K$.
\item The \textit{geometric d-polynomial hull} of $K$ is the set
  $\cP^g_d(K)$ containing all the points $z_0\in\C^n$, such that for
  every  polynomial $P$, $\deg P\le d$ and $P(z_0)=0$, there exists a $\lambda\ge
  0$, so that $\{P(z)=\lambda\}\cap K\neq\emptyset$. A compact
  set $K\subset\C^n$ is\textit{ geometrically d-polynomially convex}
  if $\cP^g_d(K)=K$
\end{itemize}
\end{definition}
\noindent Stolzenberg's result shows that $$\cP(K)=\cP_\infty(K)=\cP^g_{\infty}(K)=\cP^{Oka}_\infty (K),$$ and from
the proof of his result, we clearly see that \[\cP^{Oka}(K)\subseteq\cP_d^g(K)\subseteq \cP_d(K)\] for
every $d\in\N$. It is also obvious that for $d_1\le d_2$, we have
inclusions
\[\cP_{d_2}(K)\subseteq \cP_{d_1}(K), \cP_{d_2}^g(K)\subseteq
 \cP^g_{d_1}(K),  \cP_{d_2}^{Oka}(K)\subseteq \cP^{Oka}_{d_1}(K)\] and
 that 
\[\cP(K)=\bigcap_{d\in\N}P_d(K)=\bigcap_{p\in\N}\cP_d^g(K)=\bigcap_{d\in\N}\cP_d^{Oka}(K)\]
holds.
We will see from examples below, that in general, inclusions are
strict. 
\begin{remark}In any dimension, polynomial $1$-convexity is just the standard convexity, while Oka $1$-polynomial convexity is harder to characterise.  A set is $K\subset \C^n\subset \CP^n$ is Oka $1$-polynomially convex, if it is linearly convex and the set $K^*\subset\CP^{n*}$, defined as the set of hyperplanes in $\CP^n$ not intersecting $K$,  is the convex hull of a connected set \cite[Remark 2.1.10]{APS}. 
\end{remark}
\begin{remark}When $n=1$, it is clear that if $K\in\C$ is
already polynomially convex, then $P^{Oka}_1(K)=K$, since polynomial convexity in dimension one is equivalent to the complement of $K$ being path connected, and that for general compact sets $K\subset \C$, $\PO_1(K)=\C\backslash V$, where $V$  is the unbounded connected component of $\C\backslash K$. 
One usually needs polynomials of much higher
degrees to test the polynomial convexity with the other two definitions.
\end{remark} 

The paper is organised as follows. In the next section, we give some examples to illustrate the problem in dimension $2$. In the remaining of the paper,  we focus on dimension one, where characterising $d$-polynomial convexity already turns out to be a difficult problem. The arguments in the paper are mostly elementary. 

\section{Results in $\C^2$}

\noindent Compact subsets in a totally real subspace of $\C^2$ are always
polynomially convex and the following proposition shows that one can
already conclude that by using polynomials of degree $2$.

\begin{proposition} Let $K\subset \R^2\subset\C^2$ be a compact subset. Then
  for every $(z_0,w_0) \in\C^2\backslash K$, there exist a polynomial $P(z,w)$
  of degree two, so that $P(z_0,w_0)=1$ and $\|P\|_K<1.$ So $K=\cP(K)=\cP_2(K)$.
  \end{proposition}
  \begin{proof} Using the linear map $F(z,w)=(z+iw,z-iw)$, we can assume
    that $K\subset V$, where $V$ is the totally real plane
    $V=\{(\zeta,\bar\zeta);\zeta\in\C\}$.  Let
    $(z_0,w_0)\in\C\backslash K$ and let \[M=\max
    \{|(z-z_0)(w-w_0)|;(z,w)\in K\}.\]
    Let first 
    $(z_0,w_0)\in\C\backslash V$. Let $z_0w_0=a+ib$. If $b\ne 0$,  we set
    \[P(z,w)=1+\frac{i}{m}(zw-z_0w_0),\]
      with $m\in\R$, $m\ne 0$, to be determined.
    Then $P(z_0,w_0)=1$ and
    \[|p(\zeta,\bar\zeta)|^2=(1+\frac{b}{m})^2+\frac{1}{m^2}(|\zeta|^2-a)^2.\]
    On $K$, $|P(z,w)|<1$, as long as
    \[-2bm>b^2+(|\zeta|^2-a)^2)\]
 for $(\zeta,\bar\zeta)\in K$, which is equivalent to $-2bm>M$. 
 
 \noindent If $b=0$, we can set 
 \[P(z,w)=1+\frac{i}{m}(zw-iz_0w_0).\] 
 Again $P(z_0,w_0)=1$ and 
  \[|p(\zeta,\bar\zeta)|^2=(1+\frac{a}{m})^2+\frac{1}{m^2} |\zeta|^4.\]
  
  \noindent On $K$, $|P(z,w)|<1$, as long as
    \[-2am>a^2+|\zeta|^4\]
 for $(\zeta,\bar\zeta)\in K$. 
 
 \noindent If $(z_0,w_0)\in V\backslash K$, we can set
 \[P(z,w)=1-\frac{1}{2M}(z-z_0)(w-w_0).\]
    \end{proof}

As expected, the same does not hold for general compact sets in
$\C^2$, even if we restrict to simple curves  in the standard
Lagrangian torus. 
    
\begin{proposition} Let $p,q\in\N$, $D(p,q)=1$, and $K=\{(e^{ip\theta},e^{-iq\theta};0\le\theta\le2\pi\}\subset\C^2$ be the $(p,-q)-$torus knot. Then every polynomial $P(z,w)$, with
$P(0,0)=1$ and $\|p\|_K<1$,  must be of degree at least
$p+q$. Furthermore  $\cP_{p+q}(K)=K$ and for $d<p+q$,
  $\cP_d(K)$ strictly contains $K$.
  \end{proposition}
  \begin{proof} We can assume that $p>q$ and let $P(z,w)=1+Q(z,w,)$ be a polynomial of degree less then $p+q$, $Q(0,0)=0$, with $|1+Q(\zeta^p,\bar\zeta^q)|<1$ for $|\zeta|=1$. By substituting $\bar\zeta=\zeta^{-1}$, we get
   \[\left|1+Q(\zeta^p,\zeta^{-q})\right|<1.\]  Since $p$ and $q$ are coprime and the degree of $Q$ is less then $p+q$, $Q(\zeta^p,\zeta^{-q})$ does not contain a nonzero constant term.  Let $j\ge 0$ be the smallest integer, so that $\zeta^jQ(\zeta^p,\zeta^{-q})$ does not contain any terms of negative degree in $\zeta$.  Multiplying the inequality with $\zeta^j$, the inequality is equivalent to  
   \[\left|\zeta^j+R(\zeta)\right|<1,\quad |\zeta|=1,\]
where $R$ is a polynomial in $\zeta$, not containing the term $\zeta^j$. Let
 \[a=\frac{1}{\max_{|\zeta|=1} {|\zeta^j+R(\zeta)|}}.\] 
 Using Cauchy formula for the $j-$th derivative of
 $f$ at
 $0$, we find that this is not possible.

Strict containment for $d<p+q$ obviously follows from above. To show
 that that $\cP(K)=\cP_{p+q}(K)$, we take $(a,b)\not\in K$. If $(a,b)\notin
 S=\{z^qw^p=1\}$, we can take
 \[Q(z)=z^qw^p-1.\]
 Then $Q(K)=\{0\}$ and $|Q(a,b)|> 0$. If $(a,b)\in S\backslash K$, then the point already lies outside the (convex) standard closed polydisc, and we can separate the point by affine functions.
\end{proof}

\noindent While we see from the above proposition, that one cannot demonstrate
polynomial convexity (using the standard definition) of torus knots
$K_{p.-q}=\{(e^{ip\theta},e^{-iq\theta};0\le\theta\le2\pi\}$ with
polynomials of degree less than $p+q$, it might be possible to find
curves of surfaces of lower degree for this purpose (using Oka or geometric 
characterisation).

 \begin{example} Let
 $K_p=\{ (e^{pi\theta},e^{-i\theta});0\le\theta\le 2\pi\}\subset S^1\times S^1\subset
\C^2.$
\begin{itemize}
 \item  Let  $H_\lambda=\{(z,w)\in\C^2;zw-z-w=\lambda\}$. Then $(0,0)\in H_0$ and $K_2 \cap  H_\lambda=\emptyset$ for all $\lambda\in[0,\infty)$. 
\item  Let $H_\lambda=\{(z,w)\in\C^2;-2zw+z-2w=\lambda\}$. Then $(0,0)\in H_0$ and $K_3 \cap  H_\lambda=\emptyset$ for all $\lambda\in[0,\infty)$. 
\item Let $H_\lambda=\{(z,w)\in\C^2;3zw-3z-5w=\lambda\}$. Then $(0,0)\in H_0$ and $K_4 \cap  H_\lambda=\emptyset$ for all $\lambda\in[0,\infty)$. 
\end{itemize}
So, at least for $p\le 4$, one can test that $(0,0)$ is not in the geometric polynomial hull of $K_p$, using polynomials of degree $2$. We do not believe that this holds for larger $p$, but we are not able to find the lowest degree of polynomials for even such a simple case.  
 \end{example}
 
 \section{Points in $\C$}

\noindent The first already interesting example is understanding $n$-polynomial
convexity of points in $\C$. If a set of points has cardinality less or
equal to $n$, than it is trivially $n$ polynomially convex, since we
can always interpolate $n+1$ points with a polynomial of degree
$n$. We will see that for $n>2$, a generic configuration of $n+1$ points,
turns out to be  $n$-polynomially convex.
 
 \begin{lemma} \label{L1} Let $K=\{z_0,z_1,\ldots,z_n\}$ be a set of distinct
   points in $\C$ and let $w\not\in K$. The point $w$ is in the
   $n$-polynomial hull of $K$ if and only if all the values
   \[(z_i-w)\prod_{j\ne i}(z_i-z_j)\]
   lie on the same real half-line from the origin.
  \end{lemma} 

\begin{proof} We first assume that $w=0$ and let
  $K=\{z_0,z_1,\ldots,z_n\}$, $0\not\in K$. The point $w$ is in $n$-polynomial hull
  of $K$, if for every  polynomial
\[Q(z)=a_nz^n+a_{n-1}z^{n-1}+\ldots+a_0,\] such that
$\|Q\|_K= 1$, we also have $\|Q(0)\|\le 1$. Let us write
\[\begin{bmatrix}1&z_0&\cdots& z_0^n\\
1&z_1&\cdots &z_n^n\\
\vdots&\vdots&\ddots&\vdots\\ 
1&z_n&\cdots& z_n^n
\end{bmatrix} \begin{bmatrix}a_0\\a_1 \\
  \vdots\\a_n\end{bmatrix}=\begin{bmatrix}\alpha_0\\ \alpha_1\\
  \vdots\\ \alpha_n\end{bmatrix},\]
where
$\alpha=\begin{bmatrix}\alpha_0,&\alpha_1,&\cdots ,&\alpha_n\end{bmatrix}^T\in
\C^{n+1}$ is some point with $\|\alpha\|_\infty=1$. If we denote the Vandermonde
matrix above by $V$, we must have
\[\left|\left<\begin{bmatrix}1,&0,\cdots,&0\end{bmatrix},V^{-1}\alpha\right>\right|\le
    1.\]
This is true for all $|\alpha|_{\infty}\le 1$, if and only if, the first column of $V^{-T}$ has $l_1$ norm less or equal to $1$.  Using 
Cramer's rule, this gives the inequality
\begin{equation}\label{e1}\frac{1}{|\det V|}\sum_{i=0}^n\left|\frac{z_0z_1\cdots z_n}{z_i}\det V_i\right|\le 1,\end{equation}
where $V_i$ is the Vandermonde determinant of the 
of points $z_0,\ldots,\hat {z_{i}},,\ldots,z_n.$ Since the
determinant of a Vandermonde matrix is the product of deferences of the
points, 
\[\det V=\prod_{0\le i <j\le n}(z_j-z_i),\]
we get
\[\det V=(-1)^i\det V_i \prod_{j\ne i} (z_i-z_j).\]
By expanding the determinant of $V$ by the first column, we also get
\[\det V=z_0z_1\cdots z_n\sum_{i}(-1)^i\frac{\det V_i}{z_i}=z_0z_1\cdots z_n\det V \sum_{i}\frac{1}{z_i\prod_{j\ne i}(z_i-z_j) }.\]
The inequality (\ref{e1}) can be rewritten as
\[\sum_{i=0}^n\left|\frac{\det V}{z_i\prod_{j\ne
        i}(z_i-z_j)}\right|\le\left|\sum_{i=0}^n\frac{\det
      V}{z_i\prod_{j\ne i}(z_i-z_j)}\right|.\]
This happens, if and only if, all the values
\[z_i\prod_{j\ne i}(z_i-z_j)\]
lie on the same real half-line from the origin in $\C$.
If $w\ne 0$, we get the result by using a simple affine transformation
$z\mapsto z-w$.
\end{proof}
Using the above lemma, we can geometrically describe the
$n$-polynomial hull of $n+1$ points in $\C$.

\begin{proposition} \label{Pangle} Let $n>1$ and $K=\{z_0,z_1,\ldots,z_n\}$, where the
  points $z_0,z_1,\ldots,z_n$ are (in a counter-clockwise order)
  vertices of a convex $n+1$-sided polygon. Let $w$ be a point in the
  convex hull of $K$. For every $i=0,\ldots,n$, let $\alpha_{i,j}$ be
  the angle at the point $z_j$, $j\ne i-1,i,i+1$ ($n+1:=0$), of the
  triangle  $\Delta z_iz_{i+1}z_j$, and let $\beta_i$ be the angle at $w$ of
  the triangle $\Delta z_iz_{i+1}w$. Then  $w$ is in the
  $n$-polynomial hull of $K$, if and
  only if, for every $i=0,\ldots,n$,
  \[\beta_i+\sum_j \alpha_{i,j}=\pi.\]
  If such a point $w$ exists, it is unique.
  \end{proposition}
  \begin{proof} Let $i=0,1,\ldots,n$. From Lemma \ref{L1}, we see
    that $w$ is in the $n$-polynomial hull of $K$ if and only all the
    equalities
    \[\frac{z_i-w}{|z_i-w|}\prod_{j\ne
          i}\frac{z_j-z_i}{|z_j-z_i|}
        =\frac{z_{i+1}-w}{|z_{i+1}-w|}\prod_{j\ne
            i+1}\frac{z_j-z_{i+1}}{|z_j-z_{i+1}|}
\]
hold,
which is equivalent to
 \[\frac{(z_{i+1}-w)/|z_{i+1}-w|}{(z_{i}-w)/|z_{i}-w|}\prod_{j\ne
          i,i+1}\frac{(z_j-z_{i+1})/|z_j-z_{i+1}|}{(z_j-z_{i})/|z_j-z_{i}|}
        =-1.\]
      This is equivalent to 
      \[e^{i(\beta_i+\sum_j \alpha_{i,j})}=-1.\]
      Since the sum of all $\alpha_{i,j}$ over all $i,j$ is just the sum of the
      interior angles of the polygon, while the sum of all $\beta_i$
      is $2\pi$,  we get
      \[\sum_i\left(\beta_i+\sum_j
            \alpha_{i,j}\right)=2\pi+(n-1)\pi=(n+1)\pi,\]
        and all the summands must be equal to $\pi$.
        
 If such a point $w$ exists, it must therefore lie on the intersection
 of $n$ distinct circular arcs over the edges of the polygon, and must
 be inside the polygon (all the point with a given angle
 with respect to a line segment lie on a circle having that line
 segment as a cord, Figure \ref{F1}). 
 
 \begin{figure}[H]
 \begin{center}{\begin{tikzpicture}[scale=1.5]
  \tkzDefPoint(0,0){O}
    \tkzDefPoint(cosd(0),sind(0)){A}
  \tkzDefPoint(cosd(60),sind(60)){B}
    \tkzDefPoint(cosd(150),sind(150)){C}
    \tkzDefPoint(-0.177219,-0.984171){D}
    \tkzDefPoint(0.456756,0.381854){CCC}
     \tkzDefPoint( (2*cosd(30)-0.5/cosd(30))*cosd(30),(2cosd(30)-0.5/cosd(30))*sind(30)){AC}
      \tkzDefPoint( (2*cosd(45)-0.5/cosd(45))cosd(105),(2cosd(45)-0.5/cosd(45))*sind(105)){BC}
\tkzDefPoint( (2*cos(0.958118)-0.5/cos(0.958118))(-0.907073),(2cos(0.958118)-0.5/cos(0.958118))*(-0.420974)){CC}      
  \tkzDrawPolygon[dashed](A,B,C,D)

\tkzDrawPoints(A,B,C,D,CCC)
\tkzLabelPoint(A){$z_0$}
\tkzLabelPoint[left](B){$z_1$}
\tkzLabelPoint[left](C){$z_2$}
\tkzLabelPoint[below](D){$z_3$}
\tkzLabelPoint[right](CCC){$w$}
\tkzDrawCircle(AC,A)
\tkzDrawCircle(BC,B)
\tkzDrawCircle(CC,C)

\tkzDrawSegment(A,CCC)
\tkzDrawSegment(B,CCC)
\tkzDrawSegment(C,CCC)
\tkzDrawSegment(D,CCC)

  \end{tikzpicture}}
  \end{center}
  \caption{} 
  \label{F1}
  \end{figure}

 \noindent Since two distinct circles only intersect at most
 two
 points, and one of them is either a vertex, or lies outside the polygon, such a point $w$ must
 be unique.
\end{proof}

\begin{remark} We can write a similar condition for angles of $w$ for
  arbitrary line segments with end points in the set $K$. Let $z_iz_j$
  be a line segment. The formula we get is
  \[\beta+\sum \epsilon_k\alpha_k=\pi,\]
where $\epsilon_k$ equals to $1$ if the point $z_k$ lies in the same
half plane defined by the line through $z_i$ and $z_j$ as $w$, and $-1$
if the point lies in the opposite half plane. The
convexity condition in the proposition was only used to ensure that
all points in the same half plane defined by the line through $z_i$ and $z_{i+1}$.
\end{remark}

\begin{example} For $3$ points $z_0,z_1,z_2\in \C$,  it is easy to check that the condition  in Proposition \ref{Pangle}
  is equivalent to $w$ being the orthocenter of the triangle $\Delta z_0z_1z_2$, which has to be acute. So three points are
  $2$-polynomially convex if and only if they are the vertices of an
  obtuse or right triangle (any number of collinear points are
  $2$-polynomially convex). If we take the vertices of an acute triangle and add the orthocenter, the four point set is also $2$-polynomially convex. 
\end{example}

For convex dependent points, we have a simple result. 

\begin{proposition} If $K$ is a set of $n+1$ distinct points in $K$
  that is not convex independent, then $K$ is $n$-polynomially convex.
\end{proposition}
\begin{proof}
 We number the points in $K$ so that $\Conv K=\Conv
  \{z_0,\ldots,z_k\}$ and $z_0,\ldots,z_k$ are convex
  independent. The rest of the points $z_{k+1},\ldots,z_n$ lie in the
  interior of the convex hull of $K$. Let $w\not\in K$ be a point in
  the $n$-polynomial hull of $K$. Then $w$ lies in the interior of
  $\Conv K$, and for every $i=0,1,\ldots,k$, we have
  by Lemma \ref{L1}
  \[\beta_i+\sum_{j\le k} \alpha_{i,j}+\sum_{j> k} \alpha_{i,j}=\pi,\]
where $\beta_i$  is the angle of a triangle $\Delta z_iz_{i+1}w$ at $w$ and
$\alpha_{i,j}$ the angle at $z_j$ of the triangle $\Delta z_iz_{i+1}z_j$, $j\ne
i,i+1$  (we use $k+1=0$). Note that, since $z_0,\ldots,z_k$ are extremal points,
of $\Conv K$, all the other points are in the same side of line
segments $z_0z_1,z_1z_2,\ldots,z_kz_0$. We have
\[\sum_{i=0}^k\sum_{j\le k} \alpha_{i,j}=(k-1)\pi\]
since this is precisely the sum of interior angles of convex polygon
$z_0z_1\ldots z_k$. Since
\[\sum_{i=0}^k\sum_{j> k} \alpha_{i,j}=2(n-k)\pi\ \text{and}\ \sum_{i=1}^k\beta_i=2\pi,\]
we have
\[2\pi+(k-1)\pi+2(n-k)\pi=(k+1)\pi,\]
which is only possible if $n-k=0$.
\end{proof}
We can easily see from Lemma \ref{L1} that most configurations of $n+1$ point are $n$-polynomially convex. The set of all sets of not necessarily distinct $n+1$ points in $\C$ equals  
\[C_{n+1}=\{\{z_0,z_1,\ldots,z_n\} ,z_0,z_1,\ldots,z_n\in \C\}=\C^{n+1}/\Symm_{n+1},\] where $\Symm_{n+1}$ is the symmetric group on $n+1$ elements, and for the set of all sets of distinct $n+1$ points in $\C$, the set of unordered $n+1$-{\it configurations} in $\C$, we have
\[\UConf_{n+1}(\C)={\{ z=(z_0,z_1,\ldots,z_n)\in\C^{n+1}, z_i\ne z_j\} }/{\Symm_{n+1}}.\]
It is well known that the map $(z_0,z_1,\ldots z_n)\mapsto (z-z_0)(z-z_1)\cdots (z-z_n)$ descents to a map from $\C^{n+1}/\Symm_{n+1}$ to monic polynomials of degree $n+1$, and gives a structure of an $n+1$ dimensional algebraic manifold on both $C_{n+1}$ and $\UConf_{n+1}(\C)$. Also, $C_{n+1}$ is algebraically isomorphic to $\C^{n+1}$.
\begin{proposition} Let $n>1$. The set  $\{c\in C_{n+1}, c  \text{\ is $n$-polynomially convex} \}$ is the complement of an $n+4$ dimensional real algebraic subset of $C_{n+1}$. The set  $\{c\in\UConf_{n+1}(\C), c  \text{\ is $n$-polynomially convex} \}$ is open in $\Conf_{n+1}(\C)$, and is the complement of an $n+4$ dimensional real algebraic submanifold.
\end{proposition}
\begin{proof} Let $$S=\{(z_0,z_1,\ldots,z_n) \in \mathbb C^{n+1}; \{z_0,z_1,\ldots,z_n\}\ \text{is not polynomially convex}\}.$$ For every $Z\in S$, the coordinates of $Z$ are all distinct and there exists a unique $w_Z\in \mathbb C\backslash\{z_0,z_1,\ldots,z_n\}$, that lies in the polynomial hull of $\{z_0,z_1,\ldots,z_n\}$. In a small neighbourhood $U$ of any point in $S$, the map $\mapsto \frac{1}{z_0-w_Z}(z_1-w_Z,\cdots,z_n-w_Z)$ gives us a smooth fibration from $U$ onto an open subset of  
the set $$S'=\{(z_1,z_2,\ldots,z_n)\in\mathbb C^n; 0\ \text{is in the polynomial hull of}\ \{1,z_1,\ldots,z_n\}.$$ The fibers are $4$-dimensional. It is thus enough to check that $S'$ is an $n$ dimensional real algebraic submanifold in $\mathbb C^n$. Let every $k=1,2,\ldots,n$, 
\[F_k=\frac{(z_k-z_1)\cdots (z_k-z_{k-1})(z_k-z_{k+1})\cdots(z_k-z_n)z_k}{(1-z_1)\cdots(1-z_n)}\]
and let $F=(F_1,\ldots,F_n)$. 
From the condition in Lemma \ref{L1}, we get that near $Z'\in S'$, the point in $S'$ are solutions of $\mathrm{Im\ } F=0$. We thus need to check that the real derivative 
$D(F-\bar F)$ is of maximal rank for any $Z'\in S'$. If we denote by $\partial F$ the complex derivative of $F$, this will be true, if $\det(\partial F)\ne 0$. From
\begin{align*} \frac{\partial F_k}{\partial z_k}&=F_k\left(\frac{1}{z_k-z_1}+\cdots+\frac{1}{z_k-z_n}+\frac{1}{z_k} \right)\\
\frac{\partial F_k}{\partial z_j}&=F_k\left(\frac{1}{z_k-z_j}-\frac{1}{z_k-1}\right)\\
\end{align*}
we get
$$\det(\partial F)=\frac{F_1\cdots F_n}{z_1\cdots z_n} \begin{vmatrix}
\frac{z_1}{z_1-z_2}+\cdots+\frac{z_1}{z_1-z_n}+1&\cdots& \frac{z_1}{1-z_1}-\frac{z_1}{z_n-z_1}\\
\dots&\ddots&\vdots\\ 
\frac{z_n}{1-z_n}-\frac{z_n}{z_1-z_n}&\cdots& \frac{z_n}{z_n-z_1}+\cdots+\frac{z_n}{z_n-z_{n-1}}+1
\end{vmatrix}.$$ The last determinant is a rational function in $z_1$ and we can see that it does not have a singularity at (only possible singularities)  $z_1=z_2$ and $z_1=1$ and has a finite value at $z_1=\infty$. So the determinant does not depend on $z_1$. From symmetry, the determinant does not depend on other variables as well, so it is constant. To calculate the constant, we send the variables successively to infinity to see that it is nonzero ($n!$). 
\end{proof}

\section{Subsets on the unit circle}

Let us first look at sets of $n+1$ distinct points $z_j=e^{i\phi_j}$, $0\le \phi_0<\ldots<\phi_n<2\pi$, $j=0,\ldots,n-1$, on the unit circle. For every $j$, the angle $\angle z_jz_kz_{j+1}$ equals half of the angle $\angle z_j0z_{j+1}$,  the condition for $w$ in the unit disk from Proposition \ref{Pangle} to be in the  $n$-polynomial hull has an obvious simplification using the Thales's theorem:
\[\beta_j=\pi-(n-1)\frac{\phi_{j+1}-\phi_j}{2},\]
where $\beta_j$ is the angle $\angle \phi_{j}w\phi_{j+1}$ and the point $w$ most lie on the intersection of $n$ circles with arcs $z_jz_{j+1}$, $j=0,\ldots,n-1$, with centers 
\[c_j=\frac{\sin (n(\phi_{j+1}-\phi_j)/2)}{\sin((n-1)(\phi_{j+1}-\phi_j)/2)} \left(\cos\frac{\phi_j+\phi_{j+1}}{2},\sin\frac{\phi_j+\phi_{j+1}}{2}\right)\]
and radia
\[r_j=\frac{\sin \frac{\phi_{j+1}-\phi_j}{2}}{\sin\frac{(n-1)(\phi_{j+1}-\phi_j)}{2}}.\]
Notice that for points on the unit circle, the condition in Proposition \ref{Pangle} only has to be checked for $n$ of the angles, since it then automatically holds for the last one. 
With this notation, we get the following observation.

\begin{proposition} Let $z_j=e^{i\phi_j}$, $0\le \phi_0<\ldots<\phi_n<2\pi$ be $n+1$ distinct points o the unit circle.
\begin{itemize}
\item[(i)] If any of the angles $\phi_{j+1}-\phi_j$ is greater or equal $2\pi/n$, the set of points is $n$-polynomially convex.
\item[(ii)] If for any $i,j$, $(\phi_{j+1}-\phi_j)+(\phi_{i+1}-\phi_i)$ is less or equal $2\pi/n$, the set of points is $n$-polynomially convex.
\item[(iii)] If the set of points is not $n$-polynomially convex, then the extra point $w$ in the $n$-polynomial hull is contained in the triangle $\Delta z_j0z_{j+1}$ with the smallest angle $\angle z_j0z_{j+1}$.
\end{itemize}
\end{proposition}
\begin{proof} 
Let us assume that for some $j$, $\phi_{j+1}-\phi_j>\frac{2\pi}{n}$ and let $w$ be the extra point in the $n$-polynomial hull. The angle at $w$ of the triangle $z_jwz_{j+1}$ would have to be smaller than $\frac{\pi}{n}<\frac{\phi_{j+1}-\phi_j}{2}$ and so $w$ would need lie outside of the unit disk and thus outside of the convex hull of the points. For (ii), we only need to observe that if $(\phi_{j+1}-\phi_j)+(\phi_{i+1}-\phi_i)\le 2\pi/n$, at least one of the the other angles $\phi_{k+1}-\phi_k$ has to be greater or equal to $2\pi/n$.

To prove (iii), let us assume that $\phi_0=0$ and that $\phi_1-\phi_0$ is the smallest of the angles. The two circles with centers $c_0$ ad $c_1$ with appropriate radii $r_0$ and $r_1$  already intersect at the point $z_1$. The other point of intersection is the reflection of $z_1$ across the line segment, connecting $c_0$ and $c_1$. 

\begin{figure}[H]
\begin{center}{\begin{tikzpicture}[scale=1.5]
  \tkzDefPoint(0,0){O}
    \tkzDefPoint(cosd(0),sind(0)){A}
  \tkzDefPoint(cosd(45),sind(45)){B}
    \tkzDefPoint(cosd(145),sind(145)){C}
     
     \tkzDefPoint( cosd(45/2)*sind(3*45/2)/sind(2*45/2),sind(30/2)*sind(3*45/2)/sind(2*45/2)){AC}
      \tkzDefPoint(cosd(95)*sind(3*100/2)/sind(2*100/2),sind(95)*sind(3*100/2)/sind(2*100/2)){BC}
     
      \tkzDefPoint( cosd(45/2)*(sind(3*45/2)/sind(2*45/2)+sind(45/2)/sind(2*45/2)),sind(45/2)*(sind(3*45/2)/sind(2*45/2)+sind(45/2)/sind(2*45/2))){ACR}
      \tkzDefPoint(cosd(95)*(sind(3*100/2)/sind(2*100/2)+sind(100/2)/sind(2*100/2)),sind(95)*(sind(3*100/2)/sind(2*100/2)+sind(100/2)/sind(2*100/2))){BCR}
\tkzDrawSegment(AC,BC)
\tkzDrawSegment(A,B)
\tkzDrawSegment[dotted](B,C)
\tkzDrawSegment(O,A)
\tkzDrawSegment(O,B)
\tkzDrawSegment[dotted](O,C)
\tkzDrawPoints(O,A,B,C,AC,BC)
\tkzLabelPoint(A){$z_0$} 
\tkzLabelPoint[left](B){$z_1$}
\tkzLabelPoint[left](C){$z_2$}
\tkzLabelPoint[below](AC){$c_{0}$}
\tkzLabelPoint[below](BC){$c_1$}
\tkzLabelPoint[left](C){$z_2$}
\tkzDrawCircle(O,A)
\tkzDrawCircle[dotted](AC,ACR)
\tkzDrawCircle[dotted](BC,BCR)
\tkzInterCC(AC,ACR)(BC,BCR) \tkzGetPoints{I1}{I2}
\tkzDrawPoints(I1)
\tkzLabelPoint[left](I1){$w$}
\tkzDrawSegment[dashed](B,I1)
  \end{tikzpicture}}
  \end{center}
  \caption{}
  \label{F2}
  \end{figure}
  
\noindent The point of intersection necessarily lies below (or on) the line segment $0z_1$, since $\phi_2-\phi_1 \ge  \phi_1-\phi_0$. Analogously, the point $w$ must be above the segment $0z_0$, since $2\pi-\phi_n\ge \phi_1-\phi_0$. So $w$ must lie in the triangle $z_00z_1$.   
\end{proof}

\begin{example} One obvious example of not $n$-polynomially convex set of $n+1$ points is a set $S$ of $n+1$ equally distributed points on the unit circle. Then its $n$-polynomial hull equals $S\cup\{0\}$. Furthermore, if the smallest angle $\phi_{j+1}-\phi_j$ goes to $0$, then all the other angles $\phi_{k+1}-\phi_k$ must approach $2\pi/n$. To see this, let $\phi_{j+1}-\phi_j< \epsilon$ and let $\phi_{k+1}-\phi_k=2\pi/n-\epsilon_k$. Since the sum of all angles equals $2\pi$, we have $\epsilon_k<\epsilon$ for all $k\ne j$. So, if $w$ in the interior of the unit disk is close to the boundary, there exists a set of $S$ of $n+1$ points on the unit circle, so that $w$ is contained in the $n$-polynomial hull of $S$ with the angle at $0$ of the triangle $\Delta z_j0z_{j+1}$, that contains $w$, is small, while all the other angles are almost $2\pi/n$. 
\end{example} 

\begin{proposition} Let $A_\alpha=\{z\in S^1;\ -\alpha\le \arg z\le \alpha\}$. If $\alpha>\frac{n-1}{n}\pi$, then $A_\alpha$ is not $n$-polynomially convex.
\end{proposition}
\begin{proof} We see from the above example, that if $\alpha<2\pi/n$, we can always find a set of $n+1$ points $S\subset A_\alpha$ so that some point near the unit circle is in the $n$-polynomial hull of $S$ and thus of $A_\alpha$. 
\end{proof}

\begin{remark} We are not able to show that $A_\alpha=\{z\in S^1;\ -\alpha\le \arg z\le \alpha\}$, $\alpha\le \frac{n-1}{n}\pi$, then $A_\alpha$ is actually $n$-polynomially convex. It is not hard to see that $A_\alpha=\{z\in S^1;\ -\alpha\le \arg z\le \alpha\}$, $\alpha<\frac{\pi}{4}$, is $2$-polynomially convex. To see this, let $|\phi|\le \alpha$ and we take the polynomial $p(x)=z^2-2e^{i\phi}z+e^{2i\phi}(a^2+1)$. We can see that $| p(z) |_A< | p(re^{i\phi})|$ for $0\le r<1$, as long as $a>0$ is large enough.  So any point in the angular sector $\{|z|<1;\ -\alpha \le \arg z\le \alpha\}$ can be separated from $A$ by a quadratic polynomial.
\end{remark}

It would be nice to be able to characterise those compact  sets $A$ in $\C$ (or $\C^n$), for which there exists an $n$, so that $\mathcal P(A)=\mathcal P_n(A)$. We call such sets \emph{finitely polynomially convex}. For example.

\begin{question} Let $\gamma$ be a smooth simple Jordan arc in $\C$. Is  $\gamma$ finitely polynomially convex?
\end{question}

We are certainly not able to show finite polynomial convexity of sets by finding lemniscates of fixed degree, uniformly approximating the arc. To achieve this, we would need polynomials of larger and larger degrees. In this direction, there is a large body of research on the improvements of Hilbert lemniscate theorem. See for example \cite{A}. 

If we restrict ourselves to rectifiable simple Jordan arcs, the answer to above question is no, as it follows from the next simple example.

\begin{example} Let  $K_n$, $n=1,2,\ldots$,  be a circle with radius $\frac{1}{2n^2}$, centred at $(1/n,0)$, with open circular arcs of angle $\pi/(n+1)$ removed at each side. Let $$K=\bigcup_n K_n\cup \bigcup_n(\gamma^-_n\cup \gamma_n^+)\cup\{(0,0)\},$$ where 
$\gamma_n^\pm$ are two arcs connecting  sets $K_n$ and $K_{n+1}$ at their boundary points (see Figure \ref{F3}). 

\begin{figure}[H]
\begin{center}{\begin{tikzpicture}[scale=5]
  \tkzDefPoint(1,0){O1}
    \tkzDefPoint(1+ cosd(45)/8,sind(45)/8){A1}
  \tkzDefPoint(1+cosd(135)/8,sind(135)/8){B1}
  \tkzDrawArc(O1,A1)(B1)
    \tkzDefPoint(1+ cosd(225)/8,sind(225)/8){B11}
  \tkzDefPoint(1+cosd(315)/8,sind(315)/8){A11}
  \tkzDrawArc(O1,B11)(A11)
\tkzDefPoint(1/2,0){O2}
    \tkzDefPoint(1/2+ cosd(30)/18,sind(30)/18){A2}
  \tkzDefPoint(1/2+cosd(150)/18,sind(150)/18){B2}
  \tkzDrawArc(O2,A2)(B2)
    \tkzDefPoint(1/2+ cosd(210)/18,sind(210)/18){B22}
  \tkzDefPoint(1/2+cosd(330)/18,sind(330)/18){A22}
  \tkzDrawArc(O2,B22)(A22)
  
  \tkzDefPoint(1/3,0){O3}
    \tkzDefPoint(1/3+ cosd(22.5)/32,sind(22.5)/32){A3}
  \tkzDefPoint(1/3+cosd(157.5)/32,sind(157.5)/32){B3}
  \tkzDrawArc(O3,A3)(B3)
    \tkzDefPoint(1/3+ cosd(200.5)/32,sind(200.5)/32){B33}
  \tkzDefPoint(1/3+cosd(337.5)/32,sind(337.5)/32){A33}
  \tkzDrawArc(O3,B33)(A33)

\tkzDefPoint(1/4,0){O4}
   \tkzDefPoint(1/4+ cosd(18)/50,sind(18)/50){A4}
 \tkzDefPoint(1/4+cosd(162)/50,sind(162)/50){B4}
 \tkzDrawArc(O4,A4)(B4)
   \tkzDefPoint(1/4+ cosd(198)/50,sind(198)/50){B44}
 \tkzDefPoint(1/4+cosd(342)/50,sind(342)/50){A44}
 \tkzDrawArc(O4,B44)(A44)
 
 \tkzDefPoint(1/5,0){O5}
   \tkzDefPoint(1/5+ cosd(15)/72,sind(15)/72){A5}
 \tkzDefPoint(1/5+cosd(165)/72,sind(165)/72){B5}
 \tkzDrawArc(O5,A5)(B5)
   \tkzDefPoint(1/5+ cosd(195)/72,sind(195)/72){B55}
 \tkzDefPoint(1/5+cosd(345)/72,sind(345)/72){A55}
 \tkzDrawArc(O5,B55)(A55)
 
 \tkzDefPoint(1/6,0){O6}
   \tkzDefPoint(1/6+ cosd(12.86)/98,sind(12.86)/98){A6}
 \tkzDefPoint(1/6+cosd(167.14)/98,sind(167.14)/98){B6}
 \tkzDrawArc(O6,A6)(B6)
   \tkzDefPoint(1/6+ cosd(192.86)/98,sind(192.86)/98){B66}
 \tkzDefPoint(1/6+cosd(347.14)/98,sind(347.14)/98){A66}
 \tkzDrawArc(O6,B66)(A66)

 \tkzDrawSegment(B1,A2)
\tkzDrawSegment(B11,A22)
 \tkzDrawSegment(B2,A3)
\tkzDrawSegment(B22,A33)
 \tkzDrawSegment(B3,A4)
\tkzDrawSegment(B33,A44)
\tkzDrawSegment(B4,A5)
\tkzDrawSegment(B44,A55)
\tkzDrawSegment(B5,A6)
\tkzDrawSegment(B55,A66)

\tkzDefPoint(0,0){O}
\tkzDrawPoints(O)

\tkzDefPoint(1/7,0){A}
\tkzDrawSegment[dotted](A,O)

    \end{tikzpicture}}
  \end{center}
  \caption{}
  \label{F3}
  \end{figure}

\noindent The arcs $\gamma_n^\pm$ can be made so that $K$ is smooth except at the point $(0,0)$. Since $K_n$ contains equally spread $2n+2$ points, the point $(1/n,0)$ is in the $(2n+1)$-polynomial hull of $K$, and so $K$ is not finitely polynomially convex.  
\end{example} 

\begin{remark} Let $A\subset S^1$ be a compact set such that $\overline A=A$ ($A$ is symmetric across $x$ axis).  If $0$ is in the $n$-polynomial hull of $A$, then for any polynomial $p(z)=a_nz^n+\cdots+a_1z+1$, we have $||p||_A\ge 1=|p(0)|$ and so $||z^np(1/z)||_A=||z^n+a_1z^{n-1}+\cdots+a_n||_A\ge 1.$ So among all the monic polynomials of degree $n$, the polynomial $z^n$ has the smallest sup norm on $A$ and is thus the (unique) Chebyshev polynomial for $A$. An example of such a set is $A_\alpha=\{z\in S^1;\ -\alpha\le \arg z\le \alpha\}$, with  $\alpha>\frac{n}{n+1}\pi,$ but we could take any compact $A=\overline A\subset S^1$ that contains a set of $n+1$ equally spread points. With a different proof, this is a well known result from the theory of Chebyshev polynomials \cite[Theorem 3]{TD}. \end{remark}

\noindent {\bf Acknowledgments.} I express my gratitude to Josip Globevnik for engaging in insightful discussions on the topic, and to Thomas Mark for not only identifying the problem of $2$-polynomial convexity of three points in $\C$, but also providing a solution to the problem (personal communication).
\bibliographystyle{amsart}

\end{document}